\documentclass[12pt]{article}
\usepackage{amsmath}
\usepackage{latexsym}
\usepackage{amssymb}
\newtheorem{thm}{Theorem}[section]
\newtheorem{la}[thm]{Lemma}
\newtheorem{Defn}[thm]{Definition}
\newtheorem{Remark}[thm]{Remark}
\newtheorem{Conj}[thm]{Conjecture}
\newtheorem{prop}[thm]{Proposition}

\newtheorem{Example}[thm]{Example}
\newtheorem{Number}[thm]{\!\!}
\newenvironment{defn}{\begin{Defn}\rm}{\end{Defn}}

\newenvironment{rem}{\begin{Remark}\rm}{\end{Remark}}

\newenvironment{numba}{\begin{Number}\rm}{\end{Number}}
\newenvironment{proof}{{\noindent\bf Proof.}}%
                  {\nopagebreak\hspace*{\fill}$\Box$\medskip\par}
\newcommand{\Punkt}{\nopagebreak\hspace*{\fill}$\Box$}
\newcommand{\wb}{\overline}

\newcommand{\at}{\symbol{'100}}

\newcommand{\mto}{\mapsto}

\newcommand{\isom}{\cong}
\DeclareMathOperator{\Ad}{Ad}
\newcommand{\N}{{\mathbb N}}
\newcommand{\K}{{\mathbb K}}
\newcommand{\bO}{{\mathbb O}}
\newcommand{\Q}{{\mathbb Q}}
\newcommand{\Z}{{\mathbb Z}}
\newcommand{\cg}{{\mathfrak g}}
\newcommand{\ca}{{\mathfrak a}}
\newcommand{\rr}{{\mathfrak r}}
\newcommand{\ce}{{\mathfrak e}}
\newcommand{\cE}{{\mathcal E}}
\newcommand{\cS}{{\mathcal S}}
\DeclareMathOperator{\ad}{ad}

\newcommand{\cs}{{\mathfrak s}}
\newcommand{\ck}{{\mathfrak k}}
\newcommand{\cz}{{\mathfrak z}}

\DeclareMathOperator{\Aut}{Aut}
\newcommand{\sub}{\subseteq}
\DeclareMathOperator{\GL}{GL}
\DeclareMathOperator{\id}{id}

\DeclareMathOperator{\rad}{rad}
\DeclareMathOperator{\Rad}{Rad}

\DeclareMathOperator{\tr}{tr}

\DeclareMathOperator{\erk}{rk}
\DeclareMathOperator{\der}{der}
\DeclareMathOperator{\Inn}{Inn}

\begin{document}
$\;$\\[-32mm]
\begin{center}
{\Large\bf Elementary {\boldmath$p$}-adic Lie groups\vspace{2mm}
have finite construction rank}\\[7mm]
{\bf Helge Gl\"{o}ckner}\vspace{2mm}
\end{center}
\begin{abstract}
\hspace*{-6mm}The class of elementary totally disconnected groups
is the smallest class of totally disconnected, locally compact, second countable groups
which contains all discrete countable groups, all metrizable
pro-finite groups, and is closed under extensions
and countable ascending
unions.
To each elementary group~$G$,
a (possibly infinite) ordinal number $\erk(G)$ can be associated,
its \emph{construction rank}.
By a structure
theorem of Phillip Wesolek,
elementary $p$-adic Lie groups are among the basic building blocks
for general $\sigma$-compact $p$-adic Lie groups.
We characterize elementary
$p$-adic Lie groups in terms of the subquotients
needed to describe them.
The characterization implies
that every
elementary $p$-adic Lie group
has finite construction rank.
Structure theorems concerning general $p$-adic Lie groups
are also obtained.\vspace{2mm}
\end{abstract}
{\bf Classification:} primary 22E20;
secondary 22E35, 22E46, 22E50\\[3mm]
%
{\bf Key words:} $p$-adic Lie group, normal series,
topological simplicity, simple group, nearly simple group,
Tits core, elementary group,
scale function, rank, centre, commutator group, subquotient,
building block, finiteness\\[9mm]
\section{\!\!Introduction and statement of main results}
\noindent
By a theorem of Ph.\ Wesolek,
every $\sigma$-compact $p$-adic Lie group~$G$
admits a series
\begin{equation*}
\{1\}\lhd A_1 \lhd A_2 \lhd G
\end{equation*}
of closed subgroups such that
$G/A_2$ is finite,
$A_2/A_1$ is a finite direct product of non-discrete,
compactly generated,
topologically simple $p$-adic Lie groups,
and $A_1$ is elementary (see \cite[Corollary~1.5]{We3}).
In this sense, elementary $p$-adic Lie groups
are among the basic building blocks
for general $\sigma$-compact $p$-adic Lie groups.
It is therefore desirable to improve the understanding
of such groups.\\[2.3mm]
We recall the construction
rank of an elementary group from~\cite[3.1]{We2}.
\begin{defn}
Let $\cE_0$ be the class of all topological groups
which are countable and discrete or metrizable and pro-finite.
If~$\lambda>0$ is an ordinal number and $\cE_\mu$
has been defined for all $\mu<\lambda$,
set $\cE_\lambda:=\bigcup_{\mu<\lambda}\cE_\mu$ if $\lambda$
is a limit ordinal.
If $\lambda$ has a precursor~$\mu$,
let $\cE_\lambda$ be the class of all topological
groups such that
\begin{itemize}
\item[(a)]
$G$ has a closed normal subgroup $N$ such that
$N\in \cE_\mu$ and $G/N\in \cE_0$, or
\item[(b)]
$G=\bigcup_{n\in\N}V_n$ for open subgroups $V_n\sub G$
such that $V_n\in \cE_\mu$ for all $n\in \N$.
\end{itemize}
The \emph{construction rank} $\erk(G)$
of an elementary group~$G$ is defined as the minimum ordinal
$\lambda$ such that $G\in \cE_\lambda$.
\end{defn}
\begin{numba}
If $G$ is a topological group
and $N\lhd G$ a closed normal subgroup
such that both $N$ and $G/N$ are elementary,
then also $G$ is elementary,
with
\begin{equation}\label{umweg}
\erk(G)\leq\erk(N)+\erk(G/N)+1
\end{equation}
(see \cite[Proposition~3.5]{We2}),
using addition of ordinals.
\end{numba}

Given a prime number~$p$, let $\Q_p$
be the field of $p$-adic numbers and $\Z_p$
be the ring of $p$-adic integers.
The $p$-adic Lie groups we consider are finite-dimensional analytic
$p$-adic Lie groups as in~\cite{Ser}.
We call a $p$-adic Lie group~$G$
\emph{linear} if it
admits an injective continuous
homomorphism $G\to \GL_n(\Q_p)$ for some~$n$.
But linearity will not be assumed unless the contrary is stated.
If~$G$ is a $p$-adic Lie group, we write
$L(G):=T_1(G)$ for its Lie algebra
(the tangent space at the identity element $1\in G$).
As usual, $\Ad_G\colon G\to \Aut(L(G))$
is the adjoint representation,
with kernel~$\ker\Ad_G$.
\begin{defn}
We say that a topological group~$G$ is \emph{nearly simple}
if every closed subnormal subgroup of~$G$
is open or discrete.
\end{defn}
\begin{defn}
We say that a $p$-adic Lie group $G$ is
\emph{extraordinary} if $G$ is
nearly simple, $G=\ker \Ad_G$
(whence $L(G)$ is abelian)
but
the commutator group $S'$ is non-discrete
for each open subnormal
subgroup $S\sub G$.
\end{defn}
If $G$ is extraordinary, then $G$ is non-discrete and non-abelian, in particular.
Moreover, $G$ is not compactly generated (see Proposition~\ref{propo4}).

The author does not know whether extraordinary $p$-adic Lie
groups exist (Remark~\ref{probambi}).
Linear Lie groups cannot be extraordinary (Remark~\ref{linambi}).

Since discrete groups need not admit composition series,
also $p$-adic Lie groups need not admit composition series
(in terms of closed subnormal subgroups).
Yet, they admit series with nearly simple subquotients
(Proposition~\ref{propo1}).
Our first main result is a structure theorem
for general $p$-adic Lie groups,
which describes the nearly simple
subquotients needed to build up such groups.
\begin{thm}\label{thmA}
For every $p$-adic Lie group~$G$, there exists
a series
\begin{equation}\label{sers}
G=G_0\rhd G_1\rhd\cdots\rhd G_n=\{1\}
\end{equation}
of closed subgroups of~$G$ such that
each of the subquotients  $Q_j:=G_{j-1}/G_j$ for $j\in \{1,\ldots, n\}$
has one of the following mutually exclusive properties:
\begin{itemize}
\item[\rm(a)]
$Q_j$ is discrete;
\item[\rm(b)]
$Q_j\cong \Z_p$;
\item[\rm(c)]
$Q_j$ is extraordinary;
\item[\rm(d)]
$Q_j$ is isomorphic to a compact open subgroup
of $\Aut(\cs)$ for a simple $p$-adic Lie algebra
$\cs$; or
\item[\rm(e)]
$Q_j$ is topologically simple, compactly generated, non-discrete,
and isomorphic to an open subgroup
of $\Aut(\cs)$ for a simple $p$-adic Lie algebra~$\cs$.
\end{itemize}
A $\sigma$-compact $p$-adic Lie group $G$ is elementary if and only if no $Q_j$ satisfies~{\rm(e)}.
\end{thm}
The preceding characterization implies:
\begin{thm}\label{thmC}
Every elementary $p$-adic Lie group
has finite construction rank.
\end{thm}
\begin{rem}
If $G$ is $\sigma$-compact, then each $Q_j$ is $\sigma$-compact in Theorem~\ref{thmA}.
\end{rem}
\begin{rem}The proof shows that every series (\ref{sers}) of closed subgroups
can be refined to a series as in Theorem~\ref{thmA}.
\end{rem}
\begin{rem}
With regard to Theorem~\ref{thmA}\,(d),
we mention that, by a theorem of Chevalley,
$\Aut(\cs)$ is
an adjoint $\Q_p$-simple $\Q_p$-algebraic group
(but not necessarily isotropic),
for every simple $p$-adic Lie algebra~$\cs$.
\end{rem}
\begin{rem}\label{newwrem}
All of the $p$-adic Lie groups $Q_j$ in Theorem~\ref{thmA}\,(e) are of \emph{adjoint simple type},
i.e., isomorphic to $H(\Q_p)^\dag$ for an adjoint $\Q_p$-simple isotropic $\Q_p$-algebraic
group~$H$ (by \cite[Theorem 4.8]{We3}).
\end{rem}
The next structure theorem (and a sketch of proof) was suggested by the
referee.
\begin{thm}\label{thmref}
Let $G$ be a $p$-adic Lie group, $\cg$ be its Lie algebra
and $\rr:=\rad(\cg)$ be its radical.
Let $K:=\ker\Ad_G$ be the kernel of the adjoint representation
of~$G$ and $R$ be the kernel of the representation of $G$ on $\cs:=\cg/\rr$
induced by~$\Ad_G$. Then $K$ and $R$ are closed, topologically characteristic
subgroups of~$G$
with $K\leq R$, and the following holds:
\begin{itemize}
\item[\rm(a)]
$K$ has an abelian open subgroup~$A$. In particular, if $K$ is $\sigma$-compact,
then~$K$ is elementary of construction rank $\erk(K)\leq 2$.
\item[\rm(b)]
$R$ is the largest closed normal subgroup of~$G$ whose Lie algebra is~$\rr$.
In particular, $R$ has an open soluble subgroup.

Moreover, there is a closed, topologically characteristic subgroup~$R_0$ of~$G$
with $K\leq R_0\leq R$ such that $R_0/K$ is soluble and~$R_0$ is open in~$R$.
In particular, $R/K$ is soluble-by-discrete, and if~$R$ is $\sigma$-compact
then it is elementary of construction rank $\erk(R)\leq 2\ell+3$, where $\ell$
is the derived length of~$R_0/K$.
\item[\rm(c)]
The quotient $G/R$ has an open, topologically characteristic subgroup~$O/R$ of finite index
that splits as a direct product of the form $S_1\times\cdots\times S_d$,
such that the Lie algebra $\cs_i$ of the factor $S_i$ is a simple factor of the semisimple
Lie algebra~$\cs$. Moreover each $S_i$ is either compact, or compactly generated, topologically simple and of
adjoint simple type.
\item[\rm(d)]
If $G$ is $\sigma$-compact, then $G$ is elementary if and only if $G/R$ is compact,
in which case $\erk(G)\leq 2\ell+4$.
\end{itemize}
\end{thm}
Thus
\begin{equation*}
G\rhd O\rhd R\rhd R_0\rhd K\geq A.
\end{equation*}
In particular, Theorem~\ref{thmref}
enables series as in (\ref{sers}) to be chosen such that
all compactly generated, topologically simple
subquotients (which are among the subquotients of $O/R$) appear above the extraordinary ones (which can only appear as subquotients of~$K$).
\section{Terminology, notation and basic facts}
We set
$\N:=\{1,2,\ldots\}$ and $\N_0:=\N\cup\{0\}$.
If $G$ is a group, we write $G'$ for its commutator group
and $G^{(n)}:=(G^{(n-1)})'$
for the terms of the derived series,
with $G^{(0)}:=G$.
We write $\Aut(\cg)$ for the group
of automorphisms of
a Lie algebra~$\cg$,
and define $\ad(x)\colon \cg\to\cg$, $y\mto [x,y]$
for $x\in \cg$.
An automorphism of a topological group~$G$
is a group automorphism which is a homeomorphism.
We let $\Aut(G)$ be the group of all automorphisms of the topological group~$G$
and we write $\Inn(G):=\{I_g\colon g\in G\}$ for the group
of all inner automorphisms
$I_g\colon G\to G$, $x\mto gxg^{-1}$.
We recall the adjoint representation:
If $G$ is a $p$-adic Lie
group and $g\in G$,
then $\Ad_G(g):=L(I_g):=T_1(I_g)$
is the Lie algebra automorphism associated to the inner automorphism~$I_g$
(the tangent map of~$I_g$ at~$1$).
Then $\Ad_G\colon G\to\Aut(L(G))$ is
continuous and $L(\Ad_G)(x)=\ad(x)$ for each $x\in L(G)$
(see~\cite{Bou}).
Thus $\ad(x)=0$ for all $x\in L(G)$ and thus $L(G)$ is abelian
whenever $\Ad_G(g)=\id_{L(G)}$ for all $g\in G$
(e.g., if~$G$ is extraordinary).
If $G$ is understood,
we also write $\Ad_g:=\Ad_G(g)$ for $g\in G$.
As usual, a topological group~$G$ is called
\emph{compactly generated} if it has a compact generating set.
It is called \emph{topologically simple} if $G\not=\{1\}$
and~$G$ does not have closed normal subgroups
other than~$\{1\}$ and~$G$.
A subgroup $H\sub G$ is called
\emph{topologically characteristic}
if $\alpha(H)=H$ for each automorphism~$\alpha$
of the topological group~$G$,
while $H$
is called \emph{characteristic}
if $\alpha(H)=H$
for each automorphism $\alpha$
of $G$ as an abstract group.

As usual, the radical $\rad(\cg)$ of a finite-dimensional Lie algebra~$\cg$
is its largest soluble ideal.
The following ad hoc notion of a radical of a group (stimulated by \cite{Clu})
is a useful device for our discussions (although it may be of limited interest
elsewhere, cf.\ Remark~\ref{pathorad}).

\begin{defn}\label{adhocrad}
The \emph{radical} $\Rad(G)$
of a group~$G$ is defined as the union of all normal,
soluble subgroups of~$G$.
\end{defn}

By a totally disconnected
local field, we mean a totally disconnected,
locally compact, non-discrete topological
field~$\K$ (e.g., $\K=\Q_p$).
We shall use a fact which
is immediate from classical work by H. Bass:
\begin{la}\label{genparreau}
Let $\K$ be a totally disconnected local field,
$\bO\sub\K$ be the maximal compact subring,
$n\in \N$
and $G\sub \GL_n(\K))$ be a subgroup
satisfying {\rm(a)} and {\rm(b)}:
\begin{itemize}
\item[{\rm(a)}]
Every $g\in G$ generates a relatively compact
subgroup of $\GL_n(\K)$;
\item[{\rm(b)}]
If $V\sub\K^n$ is a $G$-invariant vector subspace,
then $V=\K^n$ or $V=\{0\}$.
\end{itemize}
Then there exists $x \in \GL_n(\K)$
such that $xGx^{-1}\sub \GL_n(\bO)$.
\end{la}
\begin{proof}
Let $\tr\colon M_n(\K)\to\K$
be the trace on the algebra of $n\times n$-matrices,
and $|.|$ be an ultrametric absolute value on~$\K$
defining its topology.
For each $g\in G$,
the subgroup $K:=\wb{\langle g\rangle}$ of $\GL_n(\K)$
is compact by hypothesis.
Hence $K$ is conjugate to a subgroup of $\GL_n(\bO)$ (see Theorem~1
in Appendix~1 to \cite[Part~II, Chapter IV]{Ser}),
say $xKx^{-1}\sub \GL_n(\bO)$.
As a consequence, $|\tr(g)|=|\tr(xgx^{-1})|\leq 1$.
Thus $\tr(G)$ is bounded,
whence $G$ is relatively compact by \cite[Corollary 1.3\,(c)]{Bas}.
Therefore $G$ is conjugate to a subgroup of $\GL_n(\bO)$
(using \cite{Ser} again).
\end{proof}
\begin{la}\label{lemrad}
For every $p$-adic Lie group~$G$, its radical
$\Rad(G)$ is a closed,\linebreak
characteristic
subgroup of~$G$
whose Lie algebra $L(\Rad(G))$ is contained in the radical $\rad(\cg)$ of the Lie algebra $\cg:=L(G)$.
Moreover, $G$ has a soluble closed normal subgroup~$N$
which is open in~$\Rad(G)$.
\end{la}
\begin{proof}
Since $G$ is Hausdorff, the closure
$\wb{N}$ is a soluble normal subgroup of~$G$ for
each soluble normal subgroup~$N$ of~$G$
(see \cite[Chapter 3, \S9, no.\,1, Corollary 1]{Bou}).
Therefore $R(G)$ is the union of all soluble, normal,
\emph{closed} subgroups of~$G$.
If $N_1$ and $N_2$ are soluble, closed normal
subgroups of~$G$, then $N_1N_2$
is soluble as $N_1N_2/N_2\cong N_1/(N_1\cap N_2)$
and $N_2$ are soluble.
Hence also $\wb{N_1N_2}$ is soluble (as just recalled).
The set $\cS$ of soluble, closed normal subgroups of~$G$
is therefore directed under inclusion.
Since every $p$-adic Lie group
has a compact, open subgroup~$U$
which satisfies an ascending chain condition
on closed subgroups,
$N\cap U$ must become stationary,
showing that $U\cap\bigcup_{N\in\cS}N$ is closed.
Being locally closed by the preceding,
the subgroup $\Rad(G):=\bigcup_{N\in\cS}N$
of~$G$ is closed.
It is immediate from the definition that $\Rad(G)$
is characteristic.
If we choose $N\in \cS$ such that $U\cap \Rad(G)=U\cap N$,
then~$N$ is open in~$\Rad(G)$.
Hence $L(\Rad(G))=L(N)$,
which is a soluble ideal (and hence contained in $\rad(\cg)$)
as~$N$ is normal and soluble (cf.\ \cite[Chapter III, \S9, no.\,2, Proposition 3]{Bou}).
\end{proof}
We also need a variant of \cite[Lemma 6.4]{Clu}.\footnote{The
lemma in~\cite{Clu} only tackles $p$-adic Lie groups
that are closed subgroups of~$\GL_n(\Q_p)$.}
\begin{la}\label{radLie}
Let $G$ be a linear $p$-adic Lie group.
Then the radical $\Rad(G)$
is soluble
and its Lie algebra $L(\Rad(G))$
coincides with the radical~$\rad(\cg)$
of $\cg:=L(G)$.
\end{la}
\begin{proof}
We may assume that $G\sub\GL_n(\Q_p)$
for some $n$ (such that the Lie group topology on~$G$
is finer than the induced topology).
Using the next lemma,
we can copy the proof of \cite[Lemma~6.4]{Clu}
to see that $L(\Rad(G))=\rad(\cg)$. For every soluble closed normal subgroup
$N\sub G$, the closure $\wb{N}$ in $\GL_n(\Q_p)$
is normal in~$\wb{G}$ and soluble.
Hence $\Rad(G)\sub \Rad(\wb{G})$.
Since $\Rad(\wb{G})$ is soluble (see \cite[p.\,220]{Clu}),
we deduce that also~$\Rad(G)$ is soluble.
\end{proof}
\begin{la}\label{opensolvable}
Let $\K$ be an infinite field, $n\in \N$ and
$G\sub \GL_n(\K)$
be a subgroup, endowed with a topology making it a topological group\,\footnote{The proof
only uses that all inner automorphisms of~$G$ are continuous.}
which has an open subgroup which is soluble $($resp., nilpotent$)$.
Then~$G$ has an open \emph{normal} subgroup
which is soluble $($resp., nilpotent$)$.
\end{la}
\begin{proof}
We can copy the proof of~\cite[Lemma~6.1]{Clu}.
\end{proof}
\begin{rem}\label{pathorad}
As we do not restrict our attention to $p$-adic Lie groups which are linear,
we have to face certain difficulties:
\begin{itemize}
\item[(a)]
The radical $\Rad(G)$
of a $p$-adic Lie group (in the ad hoc sense of Definition~\ref{adhocrad})
need not be soluble.
For example, for each $n\in\N$
we can choose a soluble group~$H_n$
with $(H_n)^{(n)}\not=\{1\}$
and define $G$ as the group
$\bigoplus_{n\in\N}H_n$ of all $(h_n)_{n\in\N}\in\prod_{n\in\N}H_n$
such that $h_n=1$ for all but finitely many~$n$
(endowed with the discrete topology).
\item[(b)]
If $G$ has a closed normal subgroup~$N$
such that $\Rad(N)$ is not soluble,
then it is not clear whether $\Rad(N)\leq \Rad(G)$.
\item[(c)]
If $\Rad(G)$ is not soluble,
then it is not clear if $G/\Rad(G)$
has trivial radical.
\item[(d)]
If $G$ is not linear,
then $L(\Rad(G))$ need not coincide with the radical of $L(G)$.
See \cite[Theorem]{EXA}
for an example with abelian Lie algebra $L(G)\not=\{0\}$
but $\Rad(G)=\{1\}$.
\end{itemize}
In the case of a $p$-adic Lie group~$G$,
the closed subgroup~$R\lhd G$ from Theorem~\ref{thmref}
may be a farther-reaching notion of radical.
\end{rem}
\begin{numba}\label{perm}
The class of elementary totally disconnected groups
has the following permanence properties:
\emph{Closed subgroups of elementary groups
and Hausdorff quotients of elementary groups are
elementary}~\cite[Theorem 3.18]{We2}.
\end{numba}
\section{Structural results on $p$-adic Lie groups}
In this section,
we provide auxiliary results on $p$-adic Lie groups
which will help us to prove the main theorems.
\begin{defn}
We say that a topological group $G$
is \emph{nearly simple} if
every closed subnormal subgroup $S\sub G$
is open in~$G$ or discrete.
\end{defn}
In this section, we prove four
propositions. The first one provides
series which can later be refined
further to obtain Theorems~A and~B.
\begin{prop}\label{propo1}
Every $p$-adic Lie group $G$ admits a
series
\begin{equation*}
G=G_0\rhd G_1\rhd \cdots\rhd  G_n=\{1\}
\end{equation*}
of closed subgroups $G_j$
such that $G_{j-1}/G_j$ is nearly
simple for all $j\in \{1,\ldots, n\}$.
\end{prop}
The second proposition
provides information on nearly simple groups.
\begin{prop}\label{propo2}
Let $G$ be a $p$-adic Lie group,
with Lie algebra $\cg:=L(G)$.
If $G$ is nearly simple, then exactly one of the following cases
occurs:
\begin{itemize}
\item[\rm(a)]
$G$ is discrete;
\item[\rm(b)]
$\dim(G)=1$ and $G$ has an open subnormal subgroup~$W$
such that $W/D\cong \Z_p$
for some discrete normal subgroup $D\sub W$;
\item[\rm(c)]
$G$ has an open normal subgroup which is extraordinary;
\item[\rm(d)]
The Lie algebra $\cg$ is simple
and $\Ad_G\colon G\to \Aut(\cg)$ has discrete kernel.
In this case, $Q:=G/\ker\Ad_G$ is a nearly simple group
such that each closed, non-trivial subnormal subgroup of~$Q$ is open in~$Q$
and $\Ad_Q\colon Q\to\Aut(\cg)$ is injective.
\end{itemize}
\end{prop}
Further information on the groups $Q$ from part (d)
of Proposition~\ref{propo2}
is provided by the next proposition.
\begin{prop}\label{propo3}
Let $G$ be a $p$-adic Lie group
with simple Lie algebra $\cg:=L(G)$.
If $\Ad_G\colon
G\to\Aut(\cg)$ is injective,
then exactly one of the cases {\rm(a)} and {\rm(b)}
occurs:
\begin{itemize}
\item[{\rm(a)}]
$G$ is isomorphic to a compact open subgroup
in $\Aut(\cg)$;
\item[{\rm(b)}]
$G$ has an open, topologically characteristic, non-compact subgroup
$M$ which is topologically simple,
non-discrete, and compactly generated.
\end{itemize}
If, moreover, $G$ is elementary, then {\rm (a)} holds.
\end{prop}
The next proposition provides vital information
on extraordinary groups.
\begin{prop}\label{propo4}
Let $G$ be an extraordinary $p$-adic Lie group.
Then $G=\bigcup_{n\in\N}G_n$
with open subgroups $G_1\sub G_2\sub\cdots$
such that each $G_n$ has an open centre.
Hence, if $G$ is $\sigma$-compact,
then $G$ is elementary
with $\erk(G)\leq 2$. Moreover, $G$ is not compactly generated.
\end{prop}
\begin{proof}
Since $\Ad_G(g)=\id_{L(G)}$, each $g\in G$
centralizes an open subgroup of~$G$
(cf.\ Proposition~8 in \cite[Chapter~III, \S4, no.\,4]{Bou}).
Let $V_1\supseteq V_2\supseteq\cdots$
be compact open subgroups
which form a base of identity neighbourhoods in~$G$.
Let $C_G(V_n)$ be the centralizer of~$V_n$ in~$G$.
Then $C_G(V_1)\subseteq C_G(V_2)\subseteq\cdots$ and
\begin{equation}\label{isunion}
G=\bigcup_{n\in \N}C_G(V_n)
\end{equation}
by the preceding.
By Baire's category theorem,
$C_G(V_N)$ is open in~$G$ for some~$N$.
After passage to a subsequence,
we may therefore assume that $C_G(V_n)$
is open in~$G$ for each $n\in \N$.
Note that $W_n:=V_n\cap C_G(V_n)$ is a compact open subgroup
of $C_G(V_n)$ which is contained in the centre $Z(C_G(V_n))$.
Hence $C_G(V_n)$ has open centre.
If~$G$ is $\sigma$-compact,
then $C_G(V_n)$ is elementary with
construction rank $\leq 1$,
since the compact open subgroup~$W_n$ is central (and hence normal)
in $C_G(V_n)$.
Hence $\erk(G)\leq 2$, by (\ref{isunion}).\\[2.3mm]
If $G$ was generated by a compact subset~$K$,
then we would have $K\sub C_G(V_n)$ for some $n\in\N$ and thus $G=C_G(V_n)$,
contradicting the fact that an extraordinary group~$G$ cannot have an open centre.
\end{proof}
Several lemmas
and observations
will help us to prove Propositions \ref{propo1}--\ref{propo3}.
\begin{la}\label{furtherla}
If a topological group $G$
is nearly simple, then also each closed subnormal
subgroup $S\sub G$
is nearly simple.
\end{la}
\begin{proof}
If $T$ is a closed subnormal
subgroup of~$S$, then $T$ also is a closed subnormal
subgroup of~$G$
and thus open in~$G$ (hence also open in~$S$)
or discrete.
\end{proof}
\begin{la}\label{openenough}
If $V$ is an open subgroup of a topological
group~$G$ and $V$ is nearly simple, then also
$G$ is nearly simple.
\end{la}
\begin{proof}
If $S\sub G$ is a closed subnormal subgroup of~$G$,
then $S\cap V$ is a closed subnormal subgroup of~$V$.
Since $V$ is nearly simple, $S\cap V$ is open in~$V$ or discrete.
If $S\cap V$ is open in~$V$, then $S$ is open in~$G$.
If $S\cap V$ is discrete, then $S$ is discrete.
\end{proof}
\begin{la}\label{nearlyquot}
Let $G$ be a topological group and $D\sub G$ be a discrete
normal subgroup. If $G$ is nearly simple, then also $G/D$ is nearly simple.
\end{la}
\begin{proof}
Let $q\colon G\to G/D$ be the canonical quotient map.
If~$S$ is a closed subnormal subgroup of~$G/D$,
then $q^{-1}(S)$ is a closed subnormal subgroup
of~$G$ and hence open (in which case also $S=q(q^{-1}(S))$
is open) or discrete, in which case also $S\cong q^{-1}(S)/D$
is discrete.
\end{proof}
\begin{la}\label{snclosed}
Let $G$ be a topological group and $S\sub G$ be a closed subnormal subgroup.
Then there exists a series $G=G_0\rhd G_1\rhd\cdots\rhd G_k=S$
such that $G_j$ is closed in~$G$ for all
$j\in \{1,\ldots, k\}$.
\end{la}
\begin{proof}
If $S\sub G$ is a closed subnormal subgroup,
we find a series $G=G_0\rhd G_1\rhd\cdots\rhd G_k=S$.
After replacing each $G_j$ with its closure,
we may assume without loss of generality that
all of $G_0,\ldots, G_k$ are closed.
\end{proof}
\begin{la}\label{clearexa}
If a $p$-adic Lie group $G$ has dimension $\dim(G)\leq 1$
or its Lie algebra $\cg:=L(G)$ is simple,
then $G$ is nearly simple.
\end{la}
\begin{proof}
If $\dim(G)=0$, then $G$ (and hence every subgroup
of $G$) is discrete.

If $\dim(G)=1$, then every closed subgroup of
$G$ has dimension~$1$ (in which case it is open)
or dimension~$0$ (in which case it is discrete).

If the Lie algebra $\cg$ is simple
and $N\sub G$ a non-discrete closed
normal subgroup,
then $L(N)$ is a non-zero ideal in $\cg$
and thus $L(N)=\cg$, entailing that $N$ is open in~$G$.
If $S\sub G$ is a non-discrete closed subnormal subgroup,
we find a series $G=G_0\rhd G_1\rhd\cdots\rhd G_k=S$
with closed subgroups $G_0,\ldots, G_k$,
by Lemma~\ref{snclosed}.
By the preceding, $G_1$ is open in~$G$ and thus $L(G_1)=\cg$
is simple. Repeating the argument, we see
that $G_j$ is open in $G_{j-1}$
for all $j\in \{1,\ldots, k\}$. Notably,
$S=G_k$ is open in~$G$. Thus $G$ is nearly simple.
\end{proof}
{\bf Proof of Proposition~\ref{propo1}.}
If $G$ is nearly simple (e.g., if $\dim(G)\leq 1$),
we can take $G_0:=G$ and $G_1:=\{1\}$.

\emph{Induction step}. Assume that $\dim(G)\geq 2$
and $G$ is not nearly simple. Then there exists
a closed subnormal
subgroup $S$ of~$G$ which is neither open in~$G$ nor discrete.
Thus $0<\dim(S)<\dim(G)$.
Since $S$ is subnormal, we find a series
$G=G_0\rhd G_1\rhd\cdots \rhd G_k=S$
of closed subgroups (see Lemma~\ref{snclosed}).
We may assume that $S$ has been chosen such that~$\dim(S)$
is maximal. Hence $G_1,\ldots, G_{k-1}$ are open
in~$G$ and thus $G_{j-1}/G_j$ is discrete
(and hence nearly simple) for all $j\in \{1,\ldots, k-1\}$.
Also $Q:=G_{k-1}/G_k$ is nearly simple.
To see this, let $q\colon G_{k-1}\to Q$ be the quotient map.
If the assertion was false, we could find a closed subnormal subgroup
$T\sub Q$ such that $0<\dim(T)<\dim(Q)$.
Thus $R:=q^{-1}(T)$ would be a closed subnormal
subgroup of $G$ with $\dim(G_k)<\dim(R)<\dim(G)$,
contradicting the maximality of~$\dim(S)$.
Since $\dim(G_k)<\dim(G)$, by induction we
have a series $G_k\rhd G_{k+1}\rhd\cdots\rhd G_n=\{1\}$
of closed subgroups with nearly simple quotients.\Punkt\\[2mm]
{\bf Proof of Proposition~\ref{propo2}.}
If~$G$ is discrete, we have~(a).
Now assume $\dim(G)\geq 1$.

\emph{Case} 1: \emph{$G$
has an open subnormal subgroup~$S$
with discrete commutator group~$S'$}.
Then $Q:=S/S'$ is an abelian
Lie group with $L(Q)=\cg$.
By Lemmas~\ref{furtherla} and~\ref{nearlyquot}, $Q$ is nearly simple.
Let $\ca\sub\cg$ be a $1$-dimensional vector subspace
and $\gamma\colon \Z_p\to Q$ be a continuous
homomorphism with $0\not=\gamma'(0)\in\ca$.
Then $A:=\gamma(\Z_p)$ is a compact Lie subgroup
of~$Q$ with $L(A)=\ca$
and $\gamma\colon \Z_p\to A$ is an isomorphism.
Since~$A$ is normal and non-discrete,
$A$ is open~in~$Q$ and so $\cg=\ca$
is $1$-dimensional.
Let $q\colon S\to Q$
be the quotient map. Then $W:=q^{-1}(A)$
is open and subnormal in~$G$
and $W/S'\cong\Z_p$, i.e.,\,(b) holds.

\emph{Case} 2: \emph{Each open subnormal subgroup of~$G$ has
non-discrete commutator group.}
We then consider $\Ad_G\colon G\to \Aut(\cg)$
and note that $\ker\Ad_G$ is a closed normal subgroup
of~$G$. Hence either $\ker\Ad_G$ is open
or $\ker\Ad_G$ is discrete.
If $S:=\ker\Ad_G$ is open,
then $\Ad_S(g)=\Ad_G(g)=\id_\cg$ for all~$g\in S$
and thus~$S$ is a nearly simple, extraordinary,
open normal subgroup of~$G$ (whence we have~(c)).
Now assume that $\ker\Ad_G$
is discrete.
Then $Q:=G/\ker\Ad_G$ is nearly simple
by Lemma~\ref{nearlyquot}.
Moreover, $\Rad(Q)$ is discrete.
To see this, suppose that the closed normal
subgroup~$\Rad(Q)$ of~$Q$ was non-discrete
(and hence open). Then~$Q$ has a soluble, open, normal
subgroup~$N$ (Lemma~\ref{lemrad}). Choose $k\in \N_0$ maximal
such that $N^{(k)}$ is non-discrete. Let $q\colon G\to Q$ be the quotient map.
Then $H:=q^{-1}(\overline{N^{(k)}})$ is a non-discrete,
closed normal subgroup of~$G$ (and thus open)
such that~$H'$ is discrete (which contradicts
the hypothesis of Case~2).
Thus indeed $\Rad(Q)$ is discrete.
Since $L(Q)=\cg$ and $\Ad_Q\circ q=\Ad_G$,
where $q$ is surjective
and the kernel is $\ker\Ad_G$ only,
we deduce that $\Ad_Q$ is injective.
In particular, $Q$ is a \emph{linear} $p$-adic Lie group.
Let~$\rr:=\rad(\cg)$ be the radical of $\cg=L(Q)$.
Since~$Q$ is a linear $p$-adic
Lie group, the discreteness of~$\Rad(Q)$
implies~$\rr=\{0\}$ (see Lemma~\ref{radLie}).
Thus~$\cg$ is semisimple.
If $S\sub Q$ is a non-trivial closed subnormal
subgroup, then Lemma~\ref{snclosed} provides a series
\[
Q=Q_0\rhd Q_1\rhd\cdots\rhd Q_k=S
\]
of closed subgroups of~$Q$.
Then $Q_j$ is open in~$Q$ or discrete,
for each $j\in \{0,\ldots, k\}$.
We show by induction on~$j$
that $Q_j$ is open in~$Q$. Since $Q_0=Q$,
this is trivial if $j=0$.
If $j\in\{1,\ldots, k\}$
and $Q_{j-1}$ is open in~$Q$,
let us show~$Q_j$ is open in~$Q$.
If not, then $Q_j$ was discrete and thus
$P:=Q_{j-1}/Q_j$ would be a group
with Lie algebra~$\cg$. Let $p\colon Q_{j-1}\to P$ be the quotient map.
Then $\Ad_P\circ \, p=\Ad_{Q_{j-1}}$ where $\Ad_{Q_{j-1}}=\Ad_Q|_{Q_{j-1}}$ is
injective. As a consequence, $p$ must be injective and thus
$Q_j=\{1\}$, contrary to~$Q_j\supseteq S\not=\{1\}$.
Hence~$Q_j$ is open in~$Q$ and thus also~$S=Q_k$ is open,
by induction.

It only remains to show that~$\cg$ is simple.
Write $\cg=\cs_1\oplus\cdots\oplus \cs_m$
with simple ideals $\cs_1,\ldots, \cs_m$.
Let $W\sub\Aut(\cg)$ be the group
of all $\phi\in \Aut(\cg)$ such that
$\phi(\cs_j)=\cs_j$ for all $j\in \{1,\ldots, m\}$.
Then~$W$ is a normal subgroup
of finite index in $\Aut(\cg)$
(complemented by a finite group
of permutations), as
follows from Corollary~1 to Proposition~2 in
\cite[Chapter~II, \S6, no.\,2]{Bou}.
As a consequence, $V:=\Ad_Q^{-1}(W)$
is an open normal subgroup in~$Q$.
We identify~$W$ with $\Aut(\cs_1)\times\cdots\times\Aut(\cs_m)$
and let $\pi_j\colon W\to\Aut(\cs_j)$ be the canonical projection.
Let $\exp$ be an exponential function for~$V$ (see
\cite[Chapter~III, \S7, no.\,2]{Bou}).
Let $k\in \{1,\ldots, m\}$.
If we had $m\geq 2$,
we could pick $j\in \{1,\ldots, m\}\setminus\{k\}$.
Then
\begin{equation*}
\Ad_V(\exp(x))(y)=e^{\ad(x)}(y)\in y+[\cs_j,\cs_k]=\{y\}
\end{equation*}
for all $x$ in a $0$-neighbourhood in~$\cs_j$ and all $y\in \cs_k$,
showing that
\begin{equation*}
(\pi_k\circ \Ad_V)(\exp(x))=\id_{\cs_k}
\end{equation*}
for all~$x$ in an identity neighbourhood in~$\cs_j$.
Thus $\ker(\pi_k\circ \Ad_V)$
would be a non-discrete closed normal
subgroup of~$V$ and hence
open in~$V$ (as $V$ is nearly simple by Lemma~\ref{furtherla}).
But then $\bigcap_{k=1}^m\ker(\pi_k\circ \Ad_V)=\ker\Ad_V$
is open, contradiction. Hence $m=1$ and~$\cg=\cs_1$ is simple.
We have all of~(d).\Punkt\\[2mm]
{\bf Proof of Proposition~\ref{propo3}.}
By Lemma~\ref{clearexa}, $G$ is nearly simple.
As the mapping $\Ad_G\colon G\to\Aut(\cg)$
is an injective continuous homomorphism
and $L(\Aut(\cg))\isom\cg$ (compare the theorem in \cite[\S5.3]{Hum}),
we see that $\Ad_G$ is \'{e}tale (i.e., a local diffeomorphism).
Being also injective,
$\Ad_G$ is an isomorphism of Lie groups onto
an open (and hence closed) subgroup of~$\Aut(\cg)$.
Let $G^\dag$ be the Tits core of~$G$, i.e.,
the subgroup generated by the contraction groups
\begin{equation*}
U_g:=\{x\in G\colon \mbox{ $g^nxg^{-n}\to 1$ as $n\to+\infty$ }\}
\end{equation*}
for $g\in G$ (which are closed since $G$ is a $p$-adic
Lie group \cite{Wan}),
as in \cite{CRW}. Then $G^\dag$ is topologically characteristic in~$G$
and hence so is its closure $M:=\wb{G^\dag}$.
In particular, $M$ is normal in~$G$.
Thus~$M$ is open in~$G$ or $M=\{1\}$
(see Proposition~\ref{propo2}\,(d)).

\emph{Case} 1: $M=\{1\}$. Then $G^\dag=\{1\}$,
whence $U_g=\{1\}$ for each $g\in G$
and thus $s_G(g)=1$, i.e., the scale function $s_G\colon G\to\N$
(as defined in \cite{Wil}) is identically~$1$
(cf.\ \cite[Proposition~3.21\,(3)]{BaW}).
Then each eigenvalue of $\Ad_G(g)$
in an algebraic closure of~$\Q_p$
has absolute value~$1$ (see \cite[Corollary~3.6]{FOR}).
This entails that $\Ad_G(g)$ is an isometry for
some choice of ultrametric norm on~$\cg$ (see Lemma~3.3 in~\cite{FOR}
and its proof). As a consequence,
the subgroup generated by $\Ad_G(g)$ is relatively compact.
Since $\Ad_G(G)$ is closed,
we deduce with Lemma~\ref{genparreau}
that $\Ad_G(G)$ is compact. We have~(a).

\emph{Case} 2: Assume that $M$ is open in~$G$.
Then $M\not=\{1\}$.
Since each non-trivial contraction group~$U_g$
is non-compact, also $M$ is non-compact.
Note that $L(M)=\cg$ is simple and $\Ad_M=\Ad_G|_M$
is injective (i.e., also
$M$ satisfies the hypotheses of Proposition~\ref{propo3}).
Since $M$ is non-compact, it does not satisfy the
conclusion of part (a) of the proposition. Hence
$\wb{M^\dag}$
has to be open in~$M$ (applying the preceding arguments to $M$
in place of $G$) and hence open in~$G$,
where $M^\dag$ is the subgroup of~$M$
generated by $M\cap U_x=U_x$ with $x\in M$.
As an auxiliary result, let us show that
\begin{equation*}
M=\wb{M^\dag}.
\end{equation*}
Since $\wb{M^\dag}$ is topologically characteristic
in~$M$, it is normal in $G$.
For each $g\in G$ and $x\in U_g$,
we have $g^nxg^{-n}\in \wb{M^\dag}$ for some
$n\in \N_0$ (as $\wb{M^\dag}$ is an identity neighbourhood in~$G$)
and thus $x\in g^{-n}\wb{M^\dag}g^n=\wb{M^\dag}$.
Thus $U_g\sub \wb{M^\dag}$, entailing that
$M=\wb{G^\dag}\sub \wb{M^\dag}$. Thus $M=\wb{M^\dag}$
indeed.

We now show that $M$ is topologically simple.
If $N\sub M$ is a non-trivial closed normal subgroup,
then $N$ is subnormal in~$G$.
Hence~$N$ is open in~$G$.
If $g\in M$ and $x\in U_g$,
then $g^nxg^{-n}\in N$ for some $n\in \N_0$.
Thus $x\in g^{-n}Ng^n=N$ (using that~$N$ is normal in~$M$).
Hence $U_g\sub N$, whence $M^\dag\sub N$ and $M=\wb{M^\dag}\sub N$.
Therefore $N=M$ and topological simplicity is established.
Finally, by  \cite[Proposition 6.5]{Clu},
there exists a group $R:=S(\Q_p)$
of $\Q_p$-rational points
of a simple isotropic $\Q_p$-algebraic group~$S$
such that
$M \cong (R^\dag)^n/Z$ for some $n\in \N$ and closed central
subgroup~$Z\sub (R^\dag)^n$. Since $R^\dag$ is compactly
generated~\cite[Chapter~I, Corollary 2.3.5]{Mar},
also~$M$ is compactly generated.

To complete the proof, recall that
an elementary group that is topologically simple
and compactly generated
must be discrete~\cite[Proposition~6.5]{We2}.
If~$M$ is as in (b) and~$G$ is elementary, then also
the closed subgroup~$M$ is elementary (see \ref{perm}).
As~$M$ is non-discrete,
this contradicts the cited proposition.\Punkt
\section{Proofs for Theorems \ref{thmA} and \ref{thmC}}
{\bf Proof of Theorem \ref{thmA}.}
We may assume that $G$ is non-discrete.
By Proposition~\ref{propo1},
$G$ admits a series $G=H_0\rhd H_1\rhd \cdots\rhd H_m=\{1\}$
of closed subgroups such that $H_{j-1}/H_j$
is nearly simple for all $j\in \{1,\ldots,m\}$.
The proof is by induction on~$m\in\N$.
If $m=1$, then $G$ is nearly
simple. Thus $L(G)$ is abelian or simple,
by Proposition~\ref{propo2}.
Case~1:
If $L(G)$ is abelian,
then Proposition~\ref{propo2}\,(b) or~(c)
yields a series of the desired form.
Case~2: If $L(G)$ is simple,
set $Q:=G/\ker\Ad_G$.
Then
Propositions~\ref{propo2}\,(d)
and~\ref{propo3}
show that $\ker\Ad_G$ is discrete
and either

(i): $Q$ is isomorphic to a compact open subgroup
in $\Aut(\cg)$ (in which case $G\rhd \ker\Ad_G\rhd \{1\}$
is as desired); or

(ii): $Q$ has an open, topologically characteristic subgroup~$M$
as in Proposition~\ref{propo3}\,(b),
in which case we let $q\colon G\to Q$ be the canonical quotient map
and take $G\rhd q^{-1}(M)\rhd \ker\Ad_G\rhd\{1\}$.

\hspace*{-.5mm}By\hspace*{-.3mm} Proposition\,\ref{propo3},
\hspace*{-.5mm}(ii) \hspace*{-.3mm}cannot \hspace*{-.3mm}occur if~$G$ is elementary, as~$Q$
is elementary by \ref{perm}.

\emph{Induction step.}
If $m>1$, we have a series
$G/H_1=P_0\rhd P_1\rhd\cdots\rhd P_k=\{1\}$
of the desired form (by the case $m=1$).
We let $q\colon G\to G/H_1$ be the quotient map
and define $G_j:=q^{-1}(P_j)$ for $j\in \{0,\ldots,k\}$.
Also, by the case $m-1$ we have a series
$H_1=G_k \rhd G_{k+1}\cdots\rhd G_n=\{1\}$
of the desired form,
for some $n>k$.
Then $G=G_0\rhd G_1\rhd\cdots\rhd G_n=\{1\}$
is as desired.\Punkt
\begin{la}\label{obviousla}
Let $G$ be an elementary group and
$G=G_0\rhd G_1\rhd\cdots\rhd G_n=\{1\}$
be a series of closed subgroups,
with $n\in \N$.
Then $G_j$ and
$Q_j:=G_{j-1}/G_j$ are elementary
for all $j\in \{1,\ldots,n\}$.
If $\erk(Q_j)$ is finite for all
$j\in \{1,\ldots,n\}$,
then also $\erk(G)$ is finite
and
$\erk(G)\leq n-1 +\sum_{j=1}^n \erk(Q_j)$.
\end{la}
\begin{proof}
$G_j$ and $Q_j$ are elementary by \ref{perm}.
If $n=1$, then $G\cong Q_1$
and thus $\erk(G)=\erk(Q_1)=n-1+\erk(Q_1)$.
If $n>1$, then $\erk(G_1)\leq n-2$\linebreak
$+\sum_{j=2}^n\erk(Q_j)$
by induction and thus $\erk(G)\leq \erk(G_1)+\erk(Q_1)+1
\leq n-1+\sum_{j=1}^n\erk(Q_j)$,
using (\ref{umweg}).
\end{proof}
{\bf Proof of Theorem~\ref{thmC}.}
Choose a series
$G=G_0\rhd G_1\rhd \cdots\rhd  G_n=\{1\}$
as in Theorem~\ref{thmA},
and set $Q_j:=G_{j-1}/G_j$ for $j\in \{1,\ldots,n\}$.
Then $\erk(Q_j)=0$
if $Q_j$ is discrete
or compact
(which holds in the cases (a), (b) and (d) of Theorem~\ref{thmA}).
Otherwise (in case~(c) of Theorem~\ref{thmA}),
$Q_j$ is extraordinary and thus (since~$Q_j$
is elementary), $\erk(Q_j)\leq 2$
by Proposition~\ref{propo4}.
Let $e$ be the number of indices~$j$
such that $G_j$ is extraordinary.
Combining the preceding estimates for $\erk(Q_j)$
with Lemma~\ref{obviousla},
we get $\erk(G)\leq n-1+2e$.\Punkt
\begin{rem}\label{probambi}
The author does not know whether
extraordinary $p$-adic Lie groups
exist. For which dimensions do they exist?
For which dimensions
can we have extraordinary $p$-adic Lie groups
which are $\sigma$-compact and hence
elementary?
\end{rem}
\begin{rem}\label{cfold}
If a $p$-adic Lie group~$G$ is extraordinary, then the radical $\Rad(S)$ is discrete
for each open subnormal subgroup $S$ of~$G$.
In fact, if $\Rad(S)$ was non-discrete (and thus open),
we could choose a soluble, open, normal subgroup~$N$ of~$S$
(by Lemma~\ref{lemrad}).
Choose $k\in \N_0$ maximal such that $N^{(k)}$ is non-discrete.
Then $H:=\overline{N^{(k)}}$ would be an open subnormal
subgroup of~$G$ such that $H'=N^{(k+1)}$
is discrete, contradicting the fact that $G$ is assumed
to be extraordinary.
\end{rem}
\begin{rem}\label{linambi}
A linear $p$-adic Lie group $G$
cannot be extraordinary because
$L(\Rad(G))$ would be the radical of~$L(G)$
in this case (by Lemma~\ref{radLie})
and thus $L(\Rad(G))=L(G)$
(as $L(G)$ is abelian).
Since $\Rad(G)$ is discrete (see Remark~\ref{cfold}),
$L(G)=\{0\}$ would follow
and thus~$G$ would be discrete
(contradicting the hypothesis that~$G$
is extraordinary).
\end{rem}
\begin{rem}\label{solurem}
\emph{If $G$ is a soluble $p$-adic Lie group
with derived length $\ell$,
then $\erk(G)\leq 2\ell-1$.}
To see this, consider
the derived series $G=G^{(0)}\rhd G^{(1)}\rhd\cdots\rhd G^{(\ell)}=\{1\}$.
We pick a compact open subgroup $U_j$
of the $p$-adic Lie group $Q_j:=G^{(j-1)}/{G^{(j)}}$.
Since $Q_j$ is abelian, $U_j$ is normal
in $Q_j$ and thus $\erk(Q_j)\leq 1$.
Hence $\erk(G)\leq \ell-1+\sum_{j=1}^\ell \erk(Q_j)\leq 2\ell-1$, by
Lemma~\ref{obviousla}.
\end{rem}
\section{Proof of Theorem~\ref{thmref}}
{\bf Proof of Theorem~\ref{thmref}.}
If $\theta\in\Aut(G)$ and $g\in G$, then
\[
\theta(g)h\theta(g)^{-1}=\theta(g\theta^{-1}(h)g^{-1})
\]
for all $h\in G$, whence $I_{\theta(g)}=\theta\circ I_g\circ \theta^{-1}$
and thus
\begin{equation}\label{reuhu}
\Ad_{\theta(g)}=L(\theta)\circ\Ad_g\circ L(\theta)^{-1}.
\end{equation}
Hence $\theta(K)\sub K$, whence $K$ is topologically characteristic in~$G$.
Since $\rr$ is characteristic in~$\cg$, we have $L(\theta)(\rr)=\rr$
and so $L(\theta)$ induces an automorphism
\begin{equation}\label{thetbar}
\wb{\theta}\in\Aut(\cs),
\end{equation}
determined by $\wb{\theta}(x+\rr):=L(\theta)(x)+\rr$ for $x\in\cg$.
As a consequence of (\ref{reuhu}), we have
\begin{equation}\label{reusecharac}
\phi(\theta(g))=\wb{\theta}\circ \phi(g)\circ\wb{\theta}^{-1}.
\end{equation}
Hence $\theta(R)\sub R$, whence $R$ is topologically characteristic in~$G$.

(a) Since $\ck:=L(K)=L(\ker\Ad_G)=\ker L(\Ad_G)=\ker\ad=\cz(\cg)$
is the centre of~$\cg$, we see that $\ck$ is abelian.
Thus~$K$ has an open abelian subgroup~$A$.
Let $V_1\supseteq V_2\supseteq \cdots$
be a descending sequence of compact open subgroups of~$K$
which form a basis of identity neighbourhoods.
If $g\in K$, then $L(I_g)=\Ad_g=\id_\cg$ implies that $I_g|_{V_n}=\id_{V_n}$
for some~$n$, whence $g$ is in the centralizer $C_K(V_n)$ of $V_n$ in~$K$.
Thus $K$ is an ascending union
\begin{equation}\label{isunio}
K=\bigcup_{n\in\N} C_K(V_n).
\end{equation}
Now assume that $K$ is $\sigma$-compact.
As $V_n\cap C_K(V_n)$ is a compact open normal subgroup of~$C_K(V_n)$,
the latter has construction rank $\leq 1$ and thus $K$ has construction rank $\erk(K)\leq 2$,
by (\ref{isunio}).

(b) We have $R=\ker(\phi)$ for $\phi\colon G\to\Aut(\cg/\rr)$, $\phi(g)(y+\rr):=\Ad_g(y)+\rr$
for $g\in G$, $y\in\cg$. Then $L(\phi)\colon \cg\to\der(\cg/\rr)$ is given by
\begin{equation*}
L(\phi)(x)(y+\rr)=[x,y]+\rr\quad\mbox{for all $x,y\in\cg$.}
\end{equation*}
Hence
\begin{equation}\label{make-eas}
L(R)=\ker L(\phi)=\{x\in\cg\colon (\forall y\in\cg)\; [x,y]\in\rr\}.
\end{equation}
It is clear from (\ref{make-eas}) that $\rr\sub L(R)$.
Since $L(R)/\rr$ is abelian by (\ref{make-eas}) and $\rr$ is soluble, also
$L(R)$ is soluble. Since $L(R)$ is an ideal, $L(R)\sub\rr$ follows and thus $L(R)=\rr$.
Hence $R$ has an open soluble subgroup.
Let $\exp\colon W\to G$ be an exponential function for~$G$
(where $W$ is an open $\Z_p$-submodule of~$\cg$)
and $e^x:=\exp(x)$ for $x\in W$.
If $N\sub G$ is a closed normal subgroup with $L(N)=\rr$,
let $g\in N$ and $x\in W$. Then
\[
f\colon\Z_p\to N,\;\; t\mto (e^{-tx}ge^{tx})g^{-1}
\]
is an analytic map with $f'(0)\in L(N)=\rr$.
Hence
\begin{equation*}
\Ad_g(x)=L(I_g)(x)=\frac{d}{dt}\Big|_{t=0}(ge^{tx}g^{-1})=\frac{d}{dt}\Big|_{t=0}(e^{tx}f(t))=
x+f'(0)\in x+\rr,
\end{equation*}
whence $g\in R$ and thus $N\sub R$.

Let $q\colon G\to G/K$, $g\mto gK$ be the canonical quotient map.
Since $\ck$ is soluble, we have $\rad(\cg/\ck)=\rr/\ck$. As the continuous homomorphism
\begin{equation*}
G/K\to\Aut(\cg), \quad gK\mto\Ad_g
\end{equation*}
is injective, $G/K$ is a linear Lie group.
Hence $\Rad(G/K)$ is a characteristic, soluble closed subgroup of $G/K$
with $L(\Rad(G/K))=\rad(\cg/\ck)=\rr/\ck$.
Now $R_0:=q^{-1}(\Rad(G/K))$ is a closed normal subgroup
of~$G$ with $L(R_0)=\rr$, whence~$R_0$ is an open subgroup of~$R$.
Moreover, $R_0/K\cong \Rad(G/K)$ is soluble. Since $\ker(q)=K$ is topologically characteristic in~$G$
and $\Rad(G/K)$ is topologically characteristic in $G/K$, it follows that $R_0=q^{-1}(\Rad(G/K))$
is topologically characteristic in~$G$.

Let $\ell$ be the derived length of $R_0/K$.
If $R$ is $\sigma$-compact,
then $\erk(R_0/K)\leq 2\ell-1$ by Remark~\ref{solurem}. Since $\erk(K)\leq 2$ by (a),
we deduce with (\ref{umweg})
that $\erk(R_0)\leq 2\ell+2$ and $\erk(R)\leq 2\ell+3$.

(c) The Levi decomposition for $\cg$ entails that $\cs:=\cg/\rr$
is a semi-simple Lie algebra. If $J_1,\ldots, J_m$ are the isotypic
components of~$\cs$, then
\begin{equation*}
\Aut(\cs)\cong \Aut(J_1)\times\cdots\times\Aut(J_m).
\end{equation*}
For $i\in\{1,\ldots, m\}$, we have $J_i\cong \ce_i^{m_i}$ with $m_i\in\N$ and a simple Lie algebra $\ce_i$. Hence $\Aut(J_i)$ is a wreath product
\begin{equation*}
\Aut(J_i)=\Aut(\ce_i)^{m_i}\rtimes S_{m_i}
\end{equation*}
with $S_{m_i}$ the symmetric group of all permutations of $\{1,\ldots, m_i\}$ (as is well known).
We write $\Aut(\cs)_\delta$ for the normal subgroup of $\Aut(\cs)$ corresponding
to
\begin{equation*}
\prod_{i=1}^m\Aut(\ce_i)^{m_i}.
\end{equation*}
Since $\Aut(\cs)_\delta$ is closed in $\Aut(\cs)\sub\GL(\cs)$ and has finite index,
we see that $\Aut(\cs)_\delta$ is open in $\Aut(\cs)$.

Now consider the continuous, injective homomorphism
\begin{equation*}
\psi\colon G/R\to \Aut(\cs),\quad gR\mto \phi(g)
\end{equation*}
induced by~$\phi$; identifying $L(G/R)$ with $\cs=\cg/\rr$,
we have $\psi=\Ad_{G/R}$. Then also the linear map
\[
L(\psi)\colon L(G/R)\to L(\Aut(\cs))
\]
is injective. Since both $L(G/R)$ and $L(\Aut(\cs))\cong\der(\cs)$
are isomorphic to~$\cs$, we see that $L(\psi)$ is an isomorphism.
As $\psi$ is both injective and \'{e}tale,
$\psi$ is an isomorphism of topological groups onto an open subgroup
of $\Aut(\cs)$. Notably, $G/R$ is $\sigma$-compact.

With identifications as above, for $i\in\{1,\ldots, m\}$ and $j\in\{1,\ldots, m_i\}$
let $H_{i,j}$ be the preimage under~$\psi$ of the $j$th copy of $\Aut(\ce_i)$ in
\begin{equation*}
\Aut(\cs)_\delta=\prod_{k=1}^m\Aut(\ce_k)^{m_k}.
\end{equation*}
Then each $H_i$ is a closed normal
subgroup in~$G/R$, whence the subgroup~$H$ generated by the
$H_{i,j}$ is a normal subgroup of~$G/R$. Since $\psi$ is injective, we see that
\begin{equation}\label{alsotop}
H=\prod_{i=1}^m\prod_{j=1}^{m_i}H_{i,j}
\end{equation}
as an internal direct product in the sense of abstract groups (without topology).

Now $G/R$ has a compact open subgroup~$U$ which is an internal
direct product
\begin{equation*}
U=\prod_{i=1}^m\prod_{j=1}^{m_i}U_{i,j}
\end{equation*}
of compact subgroups $U_{i,j}$ whose Lie algebra $L(U_{i,j})\sub L(G/R)\cong\cs$
corresponds to the $j$th copy of $\ce_i$ under our identification of $\cs$ with
$\bigoplus_{i=1}^m\ce_i^{m_i}$.
Let $g\in U_{i,j}$.
Since $U_{i,j}$ centralizes $U_{i',j'}$,
we see that $\psi(g)=\Ad_{G/R}(g)$ is the identity on $L(U_{i',j'})$
for all $(i',j')\not=(i,j)$, and thus $g\in H_{i,j}$.
Hence $U_{i,j}\sub H_{i,j}$ for all $i$ and $j$.
As a consequence, $U\sub H$ and so~$H$ is open in~$G/R$.

Since each $H_{i,j}$ (being closed in $G/R$) is $\sigma$-compact,
the open mapping theorem shows that (\ref{alsotop})
is an internal direct product also in the sense of topological groups.

Let $i,j$ be as before. By Proposition~\ref{propo3} and its proof,
$H_{i,j}$ is either compact (in which case we define $S_{i,j}:=H_{i,j}$),
or $H_{i,j}$ is non-compact, in which case
\begin{equation*}
S_{i,j}:=\wb{H_{i,j}^\dag}
\end{equation*}
is an open, topologically characteristic
subgroup which is topologically simple and compactly generated
(and thus of adjoint simple type, by \cite[Theorem 4.8]{We3}).

Let $S$ be the subgroup of $G/R$ generated by all $S_{i,j}$.
Using (\ref{alsotop}), we see that
\begin{equation}\label{alsoproS}
S=\prod_{i=1}^m\prod_{j=1}^{m_i}S_{i,j}
\end{equation}
as an internal direct product of topological groups.

For every automorphism $\theta$ of $G/R$
and $g\in G/R$, we have
\begin{equation}\label{alsodo}
\psi(\theta(g))=\psi(\Ad_{G/K}(g))=L(\theta)\circ \psi(g)\circ L(\theta)^{-1}
\end{equation}
(cf.\ (\ref{reuhu})).
If $i,j$ as before are given, then $L(\theta)$ takes the $j$th copy of $\ce_i$
to the $k$th copy for some
$k\in\{1,\ldots, m_i\}$, and we deduce with (\ref{alsodo}) that
\begin{equation*}
\theta(H_{i,j})=H_{i,k}.
\end{equation*}
Since $\theta$ restricts to an isomorphism between the topological groups
$H_{i,j}$ and $H_{i,k}$, we deduce that $\theta(S_{i,j})=S_{i,k}$.
Hence $\theta(S)\sub S$, entailing that $S$ is topologically characteristic.

If $H_{i,j}$ is not compact, we consider $\kappa_g\colon S_{i,j}\to S_{i,j}$, $h\mto ghg^{-1}$
for $g\in H_{i,j}$ and obtain a homomorphism
\begin{equation*}
\kappa\colon H_{i,j}\to \Aut(S_{i,j}),\quad g\mto\kappa_g
\end{equation*}
which is injective; this follows from the observation that the homomorphism
\begin{equation*}
H_{i,j}\to\Aut(L(S_{i,j})),
\quad
g\mto L(\kappa_g)=\psi(g)|_{L(S_{i,j})}
\end{equation*}
is injective by injectivity of~$\psi$ and the definition of~$H_{i,j}$.
Hence
\begin{equation}\label{farright}
H_{i,j}/S_{i,j}\cong \kappa(H_{i,j})/\Inn(S_{i,j})\leq \Aut(S_{i,j})/\Inn(S_{i,j}).
\end{equation}
As the quotient group on the very right in (\ref{farright})
is finite by \cite[Theorem 4.15]{We2}, we see that $S_{i,j}$ has finite index
in~$H_{i,j}$. Hence $S$ has finite index in~$H$ (and hence in~$G$),
by (\ref{alsotop}) and (\ref{alsoproS}).

To complete the proof of (c), let $O\sub G$ be the open subgroup
with $R\sub O$ and $O/R=S$, and replace the double indices $(i,j)$
with a single index.

(d) If $G$ is elementary, then it cannot have subfactors of adjoint simple type (see
\cite[Proposition~6.5]{We2}). Thus $S_{i,j}=H_{i,j}$ is compact for all $i$ and $j$ in the proof of~(c),
whence $S=O/R$ is compact. Recalling that $O$ has finite index in~$G$,
also $G/R$ will be compact.

If, conversely, $G/R$ is compact, then $G$ is an extension of the elementary group~$R$
with $\erk(R)\leq 2\ell+3$ (see (b)) by the metrizable pro-finite group $G/R$.
Hence $G$ is elementary with $\erk(G)\leq 2\ell+4$, by~(\ref{umweg}).\Punkt\\[2mm]
\noindent
We mention that $O$ is topologically characteristic in~$G$ as $O/R=S$ is topologically
characteristic in $G/R$ and $R$ is topologically characteristic in~$G$.\\[2mm]
\emph{Acknowledgements.}
I am grateful to P.-E. Caprace, C.R.E. Raja and Ph.\ Wesolek
for discussions, and to the referee for substantial comments.
The core of the text was written
during a visit to the Fields Institute (Toronto) in February 2014,
partially supported by the latter.
\bibliographystyle{amsplain}
{\bf Helge  Gl\"{o}ckner}, Universit\"at Paderborn, Institut f\"{u}r Mathematik,\\
Warburger Str.\ 100, 33098 Paderborn, Germany;\,\,
e-mail: {\tt  glockner\at{}math.upb.de}\vfill
\end{document}